\documentclass{psp}
\usepackage[latin1]{inputenc}
\usepackage[T1]{fontenc}
\usepackage{parskip}
\usepackage{amsmath,amsfonts,amssymb,amsxtra}
\usepackage{latexsym}
\usepackage{theorem}
\usepackage{graphicx,psfrag}
\usepackage{verbatim}
\usepackage{hyperref}

{\theoremstyle{plain}                       
\theorembodyfont{\itshape}
\newtheorem{Theorem}{Theorem}}

{\theoremstyle{plain}                       
\theorembodyfont{\itshape}
\newtheorem{Corollary}{Corollary}}

{\theoremstyle{plain}                       
\theorembodyfont{\itshape}
\newtheorem{Proposition}{Proposition}}

{\theoremstyle{plain}                       
\theorembodyfont{\rmfamily}
\newtheorem{Definition}{Definition}}

{\theoremstyle{plain}                       
\theorembodyfont{\rmfamily}
}

{\theoremstyle{plain}                       
\theorembodyfont{\rmfamily}
}

{\theoremstyle{plain}                       
\theorembodyfont{\rmfamily}
\newtheorem{Remark}{Remark}}

{\theoremstyle{plain}                       
\theorembodyfont{\itshape}
\newtheorem{Lemma}{Lemma}}

\def\A{\mathcal{A}}
\def\B{\mathcal{B}}
\def\C{\mathcal{C}}
\def\N{\mathbb{N}}
\def\Z{\mathbb{Z}}
\def\M{\mathcal{M}}
\def\F{\mathcal{F}}

\def\S{\mathbb{S}}
\def\a{\alpha}
\def\b{\beta}
\def\d{\delta}
\begin{document}

\title[Erratum: Coding map for a contractive Markov system]{Erratum: Coding map for a contractive Markov system}
\author{Ivan Werner\\
{\small Email: ivan\_werner@mail.ru}}
\date{August 13, 2015}
\maketitle

\begin{abstract}\noindent
An error in the proof of Lemma 2 (ii) in [I. Werner, Math. Proc. Camb. Phil. Soc. 140(2) 333-347 (2006)], which claims the absolute continuity of dynamically defined measures (DDM), is identified. This undermines the assertion of the positivity of a DDM which provides a construction for equilibrium states in [I. Werner, J. Math. Phys. 52 122701 (2011)]. 
An explicit lower bound for the DDM appearing there is computed in the case when all maps of a contractive Markov system (CMS) are contractions, the probability functions are Dini-continuous and bounded away from zero, and there exists an equilibrium state of the CMS which is absolutely continuous with respect to the initial measure. In the case of the contraction only on average, a generalized construction is shown to provide a positive set function, but it is unknown whether it gives a measure on the Borel $\sigma$-algebra, and if it did, the measure would coincide with the original DDM. 


\end{abstract}

\tableofcontents

\section{Introduction}

Recently,  the author has found an error in the last step of the proof of Lemma 2 (ii) in \cite{Wer3} (the author is grateful to Boris M. Gurevich for the invitation to give a talk at the Dynamical Systems and Statistical Physics Seminar at the Lomonosov Moscow State University  during the preparation to which the error was discovered).  However, the main result of  \cite{Wer3} (Corollary 1) is correct even under much weaker conditions on a contractive Markov system (CMS), see Theorem 5 (ii) in \cite{Wer11}, than it was requited in all  articles which used the result (openness of the Markov partition, boundedness away from zero and Dini-continuity of the probability functions). In that respect,  the main result of \cite{Wer3} is already obsolete. 

However, the constructive approach which was taken there gave rise to the technique of dynamically defined measures (DDMs)  \cite{Wer10}\cite{Wer12} (they should not be confused with measures obtained as various limits, or in general
accumulation points, in some sense of other measures on the same $\sigma$-algebra), which allows to construct equilibrium states for such random dynamical systems, and Lemma 2 (ii) in \cite{Wer3} has been the only justification so far that the constructed measure is not zero. In that respect, the article deserves some further consideration. 

The lemma was needed for Corollary 1 in \cite{Wer3} only in the special case when $\nu$ is substituted with an invariant probability measure $\mu$ (in this article, we are using the notation from \cite{Wer3}), so that $\Phi(\mu)$ becomes a shift-invariant Borel probability measure $M$ (see Proposition 1 in \cite{Wer3}), $\lambda=1/N\sum_{i=1}^N\delta_{x_i}$, and  all probability functions $p_e|_{K_{i(e)}}$ are Dini-continuous and bounded away from zero. In this case, it is quite easy to see, from the last inequality on page 343 (p. 12 in the arXiv version) in \cite{Wer3}, that even a stronger result is true if all $w_e|_{K_{i(e)}}$'s are contractions on a bounded space,  namely there exists $0<c<\infty$ such that
\[M\leq c\Phi(\lambda). \]

In this article, we show, in particular, that the same is true if the boundedness condition on the space is removed by computing an explicit lower bound for $\Phi(\lambda)$ in terms of $M$ (Corollary \ref{fc}). No openness of the Markov partition is required for that, we only assume that $M$ is an equilibrium state (which is, for example, automatically the case when the CMS is non-degenerate \cite{Wer11}). 

In the case of the contraction only on average, it is clear that the DDM is not zero if the probability functions are constant, but there is no proof for that even if the probability functions are only Lipschitz. In Section \ref{gcs}, we show that a generalized construction provides a positive set function in this case, but we do not know whether it is a measure on the Borel $\sigma$-algebra, and if it were, it would coincide with $\Phi(\lambda)$.

Partially motivated by the discovered error, the technique of the dynamically defined measures has already been studied much closer and in a more general setup \cite{Wer16}\cite{Wer15}. It seems that, in this development, a new mathematical theory is taking shape, which we call the Dynamical Measure Theory. It is an extension of the classical Measure Theory which provides mathematical methods for construction of measures from sequences of contents (in a broad sense) which need not be consistent. The development in this article seems to indicate that the theory might be only at its beginning.

\section{General method}

Now, we will specify our method for the computation of a lower bound for $\Phi(\lambda)(\Sigma)$ in this article, which, according to the general investigation in \cite{Wer15}, is quite adequate for ergodic $M$. 

Let $\phi_0$ be a probability measure on $\A_0$ and $\Phi$ be the dynamically defined outer measure resulting from $\phi_0$ (as in Definition 2 in \cite{Wer3}). Let $\Lambda$ be a $S$-invariant Borel probability measure on $\Sigma$ such that $\Lambda\ll \phi_0$ (we will denote the restriction of $\Lambda$ on $\A_0$ also by $\Lambda$). Let $Z$ be a measurable version of the Radon-Nikodym derivative $d\Lambda/ d\phi_0$. Let $\B(\Sigma)$ denote the Borel $\sigma$-algebra on $\Sigma$.

We will use the result from \cite{Wer10} that the restriction of $\Phi$ on $\B(\Sigma)$  is a $S$-invariant measure and the following easily verifiable fact, e.g. Lemma 3 (Lemma 2 in the arXiv version) in \cite{Wer16}. For $Q\in\B(\Sigma)$, let $\dot\C(Q)$ denote the set of all $(A_m)_{m\leq 0}\in\C(Q)$ such that $A_i\cap A_j=\emptyset$ for all $i\neq j\leq 0$. Then 
\[\Phi(Q)=\inf\limits_{(A_m)_{m\leq 0}\in\dot\C(Q)}\sum\limits_{m\leq 0}\phi_0\left(S^mA_m\right)\ \ \ \mbox{ for all }Q\subset\B(\Sigma).\]

In the following, we will use the usual definitions $1/0:=\infty$, $1/\infty:=0$, $\log(0):=-\infty$ and $0\log 0:=0$

\begin{Definition}
Define
\[K\left(\Lambda|\phi_0\right):=\int  \log  Zd\Lambda.\]
It is called the {\it Kullback-Leibler divergence} of $\Lambda$ with respect to $\phi_0$.
Now, define  
\[ Z^*:=\sup\limits_{m\leq 0}Z\circ S^m\ \ \ \mbox{ and}\]
\[K^*(\Lambda|\phi_0):=\int  \log  Z^*d\Lambda.\]
\end{Definition}

It is easy to see that  $\int  \log  Z^*d\Lambda$ is well defined, $K(\Lambda|\phi_0)\leq K^*(\Lambda|\phi_0)$ and $K(\Lambda|\phi_0)= K^*(\Lambda|\phi_0)$ if $\phi_0$ is $S$-invariant. Recall that $K\left(\Lambda|\phi_{0}\right)\geq 0$ (which follows immediately from the fact that $x\log x\geq x-1$ for all $x\geq 0$).

We will use the finiteness of $K^*(\Lambda|\phi_0)$ for a computation of a lower bound for $\Phi(\Sigma)$ through the following lemma (we could refer to a stronger result which is already available in \cite{Wer15},  but, in this case, a simple proof can be given to make this paper more self contained). 
\begin{Lemma}\label{lbl}
 Suppose  $K^*(\Lambda|\phi_0)<\infty$.  Then \\
(i)
\[\Phi(Q)\geq \Lambda(Q)e^{-\frac{1}{\Lambda(Q)}\int\limits_{Q}\log Z^*d\Lambda}\ \ \ \mbox{ for all } Q\in \B(\Sigma)\mbox{ such that } \Lambda(Q)>0\mbox{, and}\]
(ii)
 \[\Phi(\Sigma)\geq e^{K(\Lambda|\hat\Phi) -K^*(\Lambda|\phi_0)}\]
 where $\hat\Phi = \Phi/\Phi(\Sigma)$ (in particular, $K(\Lambda|\hat\Phi)\leq K^*(\Lambda|\phi_0)$).
\end{Lemma}
\begin{proof}
(i)  Note that  by the hypothesis $\int\log^+ Z^*d\Lambda<\infty$.
 Let  $Q\in \B(\Sigma)$ such that $\Lambda(Q)>0$ and $(A_m)_{m\leq 0}\in\dot\C(Q)$.  Let $0<\alpha<1$. Observe that, by the convexity of $x\mapsto e^{-x}$,
\begin{eqnarray*}
\sum\limits_{m\leq 0}\int\limits_{S^mA_m}Z^{1-\alpha} d\phi_0&=&\sum\limits_{m\leq 0}\int\limits_{S^mA_m}e^{-\alpha\log Z}d\Lambda\\
&\geq&\sum\limits_{m\leq 0,\Lambda(A_m)>0}\Lambda(A_m)e^{-\frac{\alpha}{\Lambda(A_m)}\int\limits_{S^mA_m}\log Zd\Lambda}\\
&\geq&\sum\limits_{m\leq 0}\Lambda(A_m)e^{-\frac{\alpha}{\sum_{m\leq 0}\Lambda(A_m)}\sum\limits_{m\leq 0}\int\limits_{S^mA_m}\log Zd\Lambda}\\
&\geq&\Lambda(Q)e^{-\frac{\alpha}{\Lambda(\bigcup_{m\leq 0}A_m)}\int\limits_{\bigcup_{m\leq 0}A_m}\log Z^*d\Lambda}.
\end{eqnarray*}
On the other hand, by the concavity of $x\mapsto x^{1-\alpha}$ or the H\"{o}lder inequality, and then by the concavity of $x\mapsto x^{\alpha}$,
\begin{eqnarray*}
\sum\limits_{m\leq 0}\int\limits_{S^mA_m}Z^{1-\alpha} d\phi_0&\leq&\sum\limits_{m\leq 0}\phi_0(S^mA_m)^\alpha\left(\int\limits_{S^mA_m}Z d\phi_0\right)^{1-\alpha}\\
&=&\sum\limits_{m\leq 0}\Lambda(A_m)^{1-\alpha}\phi_0(S^mA_m)^\alpha\\
&\leq&\left(\sum\limits_{m\leq 0}\Lambda(A_m)\right)\sum\limits_{m\leq 0}\frac{\Lambda(A_m)}{\sum_{m\leq 0}\Lambda(A_m)}\left(\frac{\phi_0(S^mA_m)}{\Lambda(A_m)}\right)^\alpha\nonumber\\
&\leq&\left(\sum\limits_{m\leq 0}\Lambda(A_m)\right)^{1-\alpha}\left(\sum\limits_{m\leq 0}\phi_0(S^mA_m)\right)^\alpha.
\end{eqnarray*}
Hence,
\[\left(\sum\limits_{m\leq 0}\Lambda(A_m)\right)^{1-\alpha}\left(\sum\limits_{m\leq 0}\phi_0(S^mA_m)\right)^\alpha\geq \Lambda(Q)e^{-\frac{\alpha}{\Lambda(\bigcup_{m\leq 0}A_m)}\int\limits_{\bigcup_{m\leq 0}A_m}\log Z^*d\Lambda}.\]
Taking the limit as $\alpha\to 1$, gives
\begin{equation}\label{frei}
\sum\limits_{m\leq 0}\phi_0(S^mA_m)\geq\Lambda(Q) e^{-\frac{1}{\Lambda(\bigcup_{m\leq 0}A_m)}\int\limits_{\bigcup_{m\leq 0}A_m}\log Z^*d\Lambda}.
\end{equation}
Therefore,
\[ \Phi(Q)\geq\Lambda(Q) e^{-\frac{1}{\Lambda(Q)}\int\log^+ Z^*d\Lambda}.\]
This implies indirectly, that  $\Lambda\ll\Phi$.

Now, let  $(A^n_m)_{m\leq 0}\in\dot\C(Q)$ for all $n\in\N$ such that $\sum_{m\leq 0}\phi_m(A^n_m)\downarrow\Phi(Q)$ as $n\to\infty$. Then, since $\sum_{m\leq 0}\phi_m(A^n_m)\geq\Phi(\bigcup_{m\leq 0}A_m^n)\geq\Phi(Q)$, $\Phi(\bigcup_{m\leq 0}A_m^n\setminus Q)\to 0$, and therefore,  $\Lambda(\bigcup_{m\leq 0}A_m^n\setminus Q)\to 0$. Hence, by splitting $\log = \log^+ -\log^-$,  the finiteness of $K^*(\Lambda|\phi_0)$ implies  that
\[\int\limits_{\bigcup_{m\leq 0}A_m^n} \log Z^*d\Lambda\to \int\limits_{Q} \log Z^*d\Lambda\ \ \ (\mbox{ as }n\to\infty).\]
Therefore, by \ref{frei},
\[ \Phi(Q)\geq\Lambda(Q) e^{-\frac{1}{\Lambda(Q)}\int\limits_{Q}\log Z^*d\Lambda}.\]
This proves (i).

(ii) By (i), for every Borel measurable parition  $(Q_k)_{1\leq k\leq n}$ of $\Sigma$, 
\[\sum\limits^n_{k=1}\Lambda\left(Q_k\right)\log\frac{\Lambda(Q_k)}{\hat\Phi(Q_k)} - \int \log Z^*d\Lambda \leq\log\Phi(\Sigma).\]
Now, using the well-known fact that the sum in the last inequality converges to $K(\Lambda|\hat\Phi)$ if one chooses a sequence of partitions which is increasing with respect to the refinement and generates the $\sigma$-algebra (e.g. Theorem 4.1 in  \cite{Sl}), it follows that
\[K(\Lambda|\hat\Phi) - \int \log Z^*d\Lambda \leq\log\Phi(\Sigma),\]
which proves (ii). 
\end{proof}

\section{Application to CMS}

We will proceed now towards the application of Lemma \ref{lbl}  for the CMS.   Let us abbreviate $\M:=(K_{i(e)},w_e,p_e)_{e\in E}$. 

\begin{Definition}
Let $x_i\in K_i$ be fixed for all $i\in N$, as in \cite{Wer3}. Let
$\{i\in N:F(M)(K_i)>0\}\subset \S\subset \{1,...,N\}$ where $F(M)$ is the measure given by $M\circ F^{-1}$. Set $\lambda':=1/|\S|\sum_{i\in \S}\delta_{x_i}$ where $|\S|$ denotes the size of $\S$ (so that $\lambda$ from \cite{Wer3} becomes a particular case of $\lambda'$).
For $m\leq n\in\mathbb{Z}$, let $\mathcal{B}_{mn}$ denote the sub-$\sigma$-algebra of $\mathcal{A}_m$ generated by cylinder sets of the form $_m[e_m,...,e_n]$, $e_i\in E$ for all $m\leq i\leq n$. Then one easily sees by the definitions of the measures that  $M|_{\mathcal{B}_{mn}}\ll \Phi_m(\lambda')|_{\mathcal{B}_{mn}}$. Define the Radon-Nikodym derivatives 
\[Z_{mn}:=\frac{M|_{\mathcal{B}_{mn}}}{\Phi_m(\lambda')|_{\mathcal{B}_{mn}}}\mbox{ and }Z^x_{mn}:=\frac{P^m_x|_{\mathcal{B}_{mn}}}{\Phi_m(\lambda')|_{\mathcal{B}_{mn}}}\mbox{ for }x\in K.\] 
\end{Definition}
Observe that, since $M$ is $S$-invariant, and  $\Phi_m(\lambda')=\Phi_0(\lambda')\circ S^m$ for all $m\leq 0$, $Z_{m(m+k)}=Z_{0k}\circ S^m$ for all $m\leq 0$ and $k\geq 0$.  Furthermore, note that $(Z_{0n}, \mathcal{B}_{0n})_{n\in\N}$ is a $\Phi_0(\lambda')$-martingale with $\int Z_{0n}d\Phi_0(\lambda')=1$ for all $n\in\N$. Hence, by Doob's Martingale Theorem, $Z_{0\infty}:=\lim_{n\to\infty}Z_{0n}$ exists $\Phi_0(\lambda')$-a.e. 
\begin{Definition}
Set
 \[K_n(M|\Phi_0(\lambda')):=\int  Z_{0n}\log  Z_{0n}d\Phi_0(\lambda').\]
It is well known that $0\leq K_n(M|\Phi_0(\lambda'))\leq K_{n+1}(M|\Phi_0(\lambda'))$ for all $n\in\N$. 
Set 
\[K(M|\Phi_0(\lambda')):=\lim_{n\to\infty}K_n(M|\Phi_0(\lambda')),\]
which is another way to define the {\it Kullback-Leibler divergence} of measures. 
\end{Definition}

Recall that $K(M|\Phi_0(\lambda'))<\infty$ implies that $M\ll\Phi_0(\lambda')$ (e.g. Example 4.5.10  in \cite{Bog} Vol. 1). Furthermore, $M\ll\Phi_0(\lambda')$ implies that $Z_{0\infty} = Z$ $\Phi_0(\lambda')$-a.e. where $Z$ denotes, from now on, a measurable version of  $dM/d\Phi_0(\lambda')$, and
 $K(M|\Phi_0(\lambda'))=\int  Z_{0\infty}\log  Z_{0\infty}d\Phi_1(\lambda') = \int  \log  Z_{0\infty}dM$.

Now, let $\Sigma_G$ be the set of all two-sided infinite paths of the CMS, i.e.
\[\Sigma_G = \{\sigma\in\Sigma|\ i(\sigma_{i+1})=t(\sigma_i)\mbox{ for all }i\in\Z\}.\]
Set
\[b:= \sup\limits_{1\leq i \leq N}\sup\limits_{x\in K_i}\sum\limits_{e\in E,\ i(e) = i}p_e(x)d\left( w_e(x_{i(e)}), x_{t(e)}\right)\]
and
\[D:=\left\{\sigma\in\Sigma_G\left|\ \lim\limits_{m\to-\infty}w_{\sigma_0}\circ ... \circ w_{\sigma_m}(x_{i(\sigma_m)})\mbox{ exists}\right.\right\}.\]
For $\sigma\in\Sigma$, let 
\begin{equation*}
    F(\sigma):=\left\{\begin{array}{cc}
    \lim\limits_{m\to-\infty}w_{\sigma_0}\circ w_{\sigma_{-1}}\circ...\circ w_{\sigma_{m}}(x_{i(\sigma_{m})})&  \mbox{if }\sigma\in D\\
     x_{t(\sigma_0)}& \mbox{ otherwise.}
     \end{array}\right.
\end{equation*}
Let $\F$ denote the $\sigma$-algebra generated by cylinder sets of the form $_m[e_m,...,e_0]$, $e_m,...,e_0\in E$, $m\leq 0$. Set
\[E(\M):=\left\{\Lambda\in P_S(\Sigma)|\  \Lambda(D) = 1\mbox{ and }E_\Lambda(1_{_1[e]}|\F)=p_e\circ F\  \Lambda\mbox{-a.e. for all }e\in E\right\}.\]
By Corollary 1 (ii) in \cite{Wer11}, $M\in E(\M)$ if $\M$ is contractive, uniformly continuous and non-degenerate (see \cite{Wer11} for the definition of the non-degeneracy,  e.g. $\M$ is non-degenerate if all $K_i$' are open, but it also admits a large class of systems with proper Borel-measurable partitions). This already contains the main result from \cite{Wer3} that $M(D) = 1$ (by Theorem 5 (ii) in \cite{Wer11}, $M(D) = 1$ even under much weaker conditions), but we are now concerned with a proof that the outer measure $\Phi(\lambda')$ constructed in \cite{Wer3}, which was shown to define a $S$-invariant Borel measure in \cite{Wer10},  is not zero. 
\begin{Remark}\label{esr}
Note that for every $\Lambda\in E(\M)$, $1_{K_{i(\sigma_1)}}\circ F(\sigma) = 1$ for $\Lambda$-a.e. $\sigma\in\Sigma$, as
\begin{eqnarray*}
    &&\int 1_{K_{i(\sigma_1)}}\circ F(\sigma) d\Lambda(\sigma)=\sum\limits_{j=1}^N\sum\limits_{e\in E, i(e)=j}\int 1_{_1[e]}1_{K_j}\circ Fd\Lambda =\sum\limits_{j=1}^N\sum\limits_{e\in E, i(e)=j}\int p_e\circ F1_{K_j}\circ Fd\Lambda\\
&& =\sum\limits_{j=1}^N\int 1_{K_j}\circ Fd\Lambda = 1.
\end{eqnarray*}
Therefore, for continuous $w_e|_{K_{i(e)}}$'s, 
\[w_{\sigma_0}\circ ... \circ w_{\sigma_{m+1}}\circ F(S^{m}\sigma) = F(\sigma)\mbox{ for }\Lambda\mbox{-a.e. }\sigma\in\Sigma\mbox{ and all }m\leq 0.\]
\end{Remark}

Set \[C:=\sum_{j=1}^N\int_{K_j}d(x,x_j)d\mu(x),\]
\[\Delta(t):=\sup\limits_{e\in E}\sup\limits_{x,y\in K_{i(e)},\ d(x,y)\leq t}\left|p_e(x)-p_e(y)\right|\mbox{ for all }t\geq 0\mbox{, and}\]
\[d:=\sup\limits_{e\in E}d(w_ex_{i(e)}, x_{t(e)}).\]
Note that, by Lemma 14 in \cite{Wer11}, $C<b/(1-a)$ (clearly, $b\leq d<\infty$, since $E$ is finite). In the case when all $w_e|_{K_{i(e)}}$ are contraction, this can be strengthened to the following.

\begin{Lemma}\label{ucl}
If all $w_e|_{K_{i(e)}}$ are contractions with a contraction rate $0<a<1$, then
    \[d(F(\sigma),x_{i(\sigma_1)})\leq\frac{d}{1-a}\ \ \ \mbox{ for all }\sigma\in\Sigma_G.\]
\end{Lemma}
\begin{proof}
Let $\sigma\in\Sigma_G$. Let us abbreviate 
\[X_m(\sigma):= w_{\sigma_0}\circ w_{\sigma_{-1}}\circ...\circ w_{\sigma_{m}}(x_{i(\sigma_{m})}) \mbox{ for all }m\leq 0\mbox{ and } X_1(\sigma):=x_{i(\sigma_{1})}.\] 
Then, for $m\leq 0$,
\begin{eqnarray*}
    d(X_m(\sigma),x_{i(\sigma_1)})\leq\sum\limits_{m\leq k\leq 0}d(X_k(\sigma),X_{k+1}(\sigma))\leq\sum\limits_{m\leq k\leq 0}a^{-k}d.
\end{eqnarray*}
Thus taking the limit, as $m\to-\infty$, implies the assertion. 
\end{proof}

In the following theorem, we also contemplate strengthening the continuity of the probability functions.

\begin{Theorem}\label{acms}
 Suppose $\M$ is contractive with a contraction rate $0<a<1$,  $p_e|_{K_{i(e)}}$'s are Dini-continuous, and there exists $\delta>0$ such that $p_e|_{K_{i(e)}}\geq\delta$ for all $e\in E$.  Suppose $M\in E(\M)$. Then the following holds true.
 
(i) For every $A\in\A_0$ and $0<q<1$,
\[\int\limits_A\log ZdM\leq M(A)\log |\S|+\frac{M(A)}{\delta}\left[\frac{-1}{q\log\frac{1}{a}}\log M(A)+\frac{1}{1-a^q}+\sum\limits_{i=0}^\infty\Delta\left(a^{(1-q)i}C\right)\right].\]

(ii) Suppose there exist $0<q<1$ and $\a>0$ such that $\sum_{i=0}^\infty\Delta\left(\frac{1}{a^q}(i+1)^\a a^{(1-q)i}C\right)<\infty$. Then
\begin{eqnarray*}
&&\int\limits_A\log ZdM\\
&\leq&M(A)\log|\S|+\frac{M(A)}{\delta}\left[\frac{\a}{q\log\frac{1}{a}}W\left(\frac{1}{M(A)^{\frac{1}{\a}}}\right)+\frac{2-a^q}{1-a^q}+\sum\limits_{i=0}^\infty\Delta\left(\frac{1}{a^q}(i+1)^\a a^{(1-q)i}C\right)\right]
\end{eqnarray*}
for all $A\in\A_0$ such that $M(A)>0$,	where $W$ denotes the principal branch of the Lambert function.

(iii) If all $w_e|_{K_{i(e)}}$ are contractions with a contraction rate $0<a<1$, then 
\[\log Z\leq\log |\S|+\frac{1}{\delta}\sum\limits_{i=0}^\infty\Delta\left(a^{i}\frac{d}{1-a}\right)\ \ \ M\mbox{-a.e.}\]
\end{Theorem}
\begin{proof}
Observe that the hypothesis implies that $w_e|_{K_{i(e)}}$'s are Lipschitz. Also, note that $Z_{0\infty} = Z_{1\infty}\circ S^{-1}$ $M$-a.e.

(i) Let $n\in\mathbb{N}$, and $(e_1,...,e_n)$ be a path. Observe that, by the convexity of $t\mapsto t\log t$,
\begin{eqnarray*}
    &&M\left(\L _1[e_1,...,e_n]\right)\log\frac{M\left(\L _1[e_1,...,e_n]\right)}{\frac{1}{|\S|}P^1_{x_{i(e_1)}}\left(\L _1[e_1,...,e_n]\right)}\\
&=&\int P^1_x\left(\L _1[e_1,...,e_n]\right)d\mu(x)\log\frac{\int P^1_x\left(\L _1[e_1,...,e_n]\right)d\mu(x)}{\frac{1}{|\S|}P^1_{x_{i(e_1)}}\left(\L _1[e_1,...,e_n]\right)}\\
&\leq&\int P^1_x\left(\L _1[e_1,...,e_n]\right)\log\frac{ P^1_x\left(\L _1[e_1,...,e_n]\right)}{\frac{1}{|\S|}P^1_{x_{i(e_1)}}\left(\L _1[e_1,...,e_n]\right)}d\mu(x).
\end{eqnarray*}
Hence,
\[\int\limits_{_1[e_1,...,e_n]}\log Z_{1n} dM\leq\int \int\limits_{_1[e_1,...,e_n]}\log Z^x_{1n}dP^1_xd\mu(x)\]
for all $_1[e_1,...,e_n]\in\mathcal{B}_{1n}$, and therefore,
\begin{equation}\label{fce}
\int\limits_{A}\log Z_{1n} dM\leq\int \int\limits_{A}\log Z^x_{1n}dP^1_xd\mu(x)
\end{equation}
for all $A\in\mathcal{B}_{1n}$. Let  $A\in\mathcal{B}_{1n}$. By Proposition 1 in \cite{Wer3} and Theorem 5 (ii) in \cite{Wer11} (Theorem 3.27 (ii) in the  DSDC version), $\mu = F(M)$. Then, since $M\in E(\M)$, by Lemma 4 (ii) (Lemma 3.8 (ii) in the  DSDC version) in \cite{Wer11},
\begin{eqnarray*}
  \int\limits_{A}\log Z_{1n} dM &\leq& \int\limits_{A}\log Z^{F(\sigma)}_{1n}(\sigma)dM(\sigma)\\
&=&  \int\limits_{A}\log\frac{p_{\sigma_1}(F(\sigma))...p_{\sigma_n}\left(w_{\sigma_{n-1}}\circ...\circ w_{\sigma_{1}}\circ F(\sigma)\right)}{\frac{1}  {|\S|}p_{\sigma_1}(x_{i(\sigma_1)})...p_{\sigma_n}\left(w_{\sigma_{n-1}}\circ...\circ w_{\sigma_{1}}x_{i(\sigma_1)}\right)}dM(\sigma)\\
    &=&M(A)\log|\S|+\sum\limits_{i=1}^n\int\limits_{A}\log\frac{p_{\sigma_i}\left(w_{\sigma_{i-1}}\circ...\circ w_{\sigma_{1}}\circ F(\sigma)\right)}{p_{\sigma_i}\left(w_{\sigma_{i-1}}\circ...\circ w_{\sigma_{1}}x_{i(\sigma_1)}\right)}dM(\sigma).\\
\end{eqnarray*}
Using the inquality $\log(x)\leq x-1$, it follows that
\begin{eqnarray*}
    &&\int\limits_{A}\log Z_{1n} dM\\
    &\leq&M(A)\log |\S| \\
    &&+\sum\limits_{i=1}^n\int\limits_{A}\frac{\left|p_{\sigma_i}\left(w_{\sigma_{i-1}}\circ...\circ w_{\sigma_{1}}\circ F(\sigma)\right)-p_{\sigma_i}\left(w_{\sigma_{i-1}}\circ...\circ w_{\sigma_{1}}x_{i(\sigma_1)}\right)\right|}{p_{\sigma_i}\left(w_{\sigma_{i-1}}\circ...\circ w_{\sigma_{1}}x_{i(\sigma_1)}\right)}dM(\sigma)\\
     &\leq&M(A)\log |\S| \\
     &&+\frac{1}{\d}\int\limits_{A}\sum\limits_{i=1}^n\left|p_{\sigma_i}\left(w_{\sigma_{i-1}}\circ...\circ w_{\sigma_{1}}\circ F(\sigma)\right)-p_{\sigma_i}\left(w_{\sigma_{i-1}}\circ...\circ w_{\sigma_{1}}x_{i(\sigma_1)}\right)\right|dM(\sigma)
\end{eqnarray*}
for all $A\in\mathcal{B}_{1n}$. For  $\sigma\in\Sigma$, define
\begin{equation}\label{sop}
f(\sigma):=\sum\limits_{i=1}^\infty\left|p_{\sigma_i}\left(w_{\sigma_{i-1}}\circ...\circ w_{\sigma_{1}}\circ F(\sigma)\right)-p_{\sigma_i}\left(w_{\sigma_{i-1}}\circ...\circ w_{\sigma_{1}}x_{i(\sigma_1)}\right)\right|.
\end{equation}
Then
\begin{equation}\label{dcub}
   \int\limits_{A}\log Z_{1n} dM\leq M(A)\log |\S| 
    +\frac{1}{\delta}\int\limits_{A}fdM
\end{equation}
for all $A\in\mathcal{B}_{1n}$. 

We show that $f\in \mathcal{L}^1(M)$.  Using Lemma 4 (ii)  in \cite{Wer11} (Lemma 3.8 (ii) in the  DSDC version), and the contraction on average,
\begin{eqnarray}\label{ecp}
  && \int d\left(w_{\sigma_{i}}\circ...\circ w_{\sigma_{1}}\circ F(\sigma), w_{\sigma_{i}}\circ...\circ w_{\sigma_{1}}x_{i(\sigma_1)}\right)dM(\sigma)\nonumber\\
&\leq& \int\int d\left(w_{\sigma_{i}}\circ...\circ w_{\sigma_{1}}x, w_{\sigma_{i}}\circ...\circ w_{\sigma_{1}}x_{i(\sigma_1)}\right)dP^1_x(\sigma)d\mu(x)\nonumber\\
&\leq&a^iC
\end{eqnarray}
for all $i\in\N$. Let $0<q<1$. For $i\in\N\cup\{0\}$, set
\[A_i:=\left\{\sigma\in\Sigma|\ d\left(w_{\sigma_{i}}\circ...\circ w_{\sigma_{1}}\circ F(\sigma), w_{\sigma_{i}}\circ...\circ w_{\sigma_{1}}x_{i(\sigma_1)}\right)>a^{(1-q)i}C\right\}\]
Then, by \eqref{ecp}, 
\[M(A_i)\leq a^{qi}\mbox{ for all }i\geq 0.\]
Hence, for every $A\in\B(\Sigma)$,
\begin{eqnarray*}
    \int\limits_A fdM&\leq&\sum\limits_{i=1}^\infty M\left(A\cap A_{i-1}\right)+M(A)\sum\limits_{i=1}^\infty\Delta\left(a^{(1-q)(i-1)}C\right)\\
    &\leq&\sum\limits_{i=0}^\infty \min\left\{M(A),a^{qi}\right\}+M(A)\sum\limits_{i=0}^\infty\Delta\left(a^{(1-q)i}C\right).
\end{eqnarray*}
Observe that the case $M(A)\leq a^{qi}$ is equivalent to $i\leq \log M(A)/(q\log a)$. Hence,
\begin{eqnarray*}
	\int\limits_A fdM&\leq&\frac{1}{q\log a}M(A)\log M(A)+M(A)+\sum\limits_{i>\frac{\log M(A)}{q\log a}} a^{qi}+M(A)\sum\limits_{i=0}^\infty\Delta\left(a^{(1-q)i}C\right)\\
	&\leq&\frac{1}{q\log a}M(A)\log M(A)+\frac{M(A)}{1-a^q}+M(A)\sum\limits_{i=0}^\infty\Delta\left(a^{(1-q)i}C\right)
\end{eqnarray*}
for all $A\in\B(\Sigma)$. In particular, by the hypothesis, $f\in\mathcal{L}^1(M)$. Furthermore, by \eqref{dcub},
\[\log Z_{1n}\leq\log |\S| +\frac{1}{\delta}E_M(f|\mathcal{B}_{1n})\ \ \ \ M\mbox{-a.e.}\]
Therefore, since $\mathcal{B}_{1n}\uparrow\A_1$  ($n\to\infty$),
\begin{equation*}
  \log Z_{1\infty}\leq\log |\S| +\frac{1}{\delta}E_M(f|\A_1)\ \ \ \ M\mbox{-a.e.}
\end{equation*}
By Shiryaev's Local Absolute Continuity Theorem (e.g. Theorem 2, p. 514 in \cite{Shi}), this implies that $M\ll\Phi_1(\lambda')$ and $Z_{1\infty}=d\Phi_1(\lambda')/dM$. 
Thus, since $Z_{0\infty}=Z_{1\infty}\circ S^{-1}$, 
\begin{equation}\label{fre}
\log Z_{0\infty}\leq\log |\S| +\frac{1}{\delta}E_M(f|\A_0)\ \ \ \ M\mbox{-a.e.},
\end{equation}
which implies (i).

(ii) Let $\b>0$. For $i\in\N\cup\{0\}$, define
\[B_i:=\left\{\sigma\in\Sigma|\ d\left(w_{\sigma_{i}}\circ...\circ w_{\sigma_{1}}\circ F(\sigma), w_{\sigma_{i}}\circ...\circ w_{\sigma_{1}}x_{i(\sigma_1)}\right)>\b(i+1)^\a a^{(1-q)i}C\right\}.\]
Then, by \eqref{ecp},
\[M\left(B_i\right)\leq\frac{1}{\b(i+1)^\a}a^{qi}\ \ \ \mbox{ for all }i\geq 0.\]
Hence, for every $A\in\B(\Sigma)$ with $M(A)>0$,
\begin{eqnarray*}
	\int\limits_A fdM&\leq&\sum\limits_{i=1}^\infty M\left(A\cap B_{i-1}\right)+M(A)\sum\limits_{i=1}^\infty\Delta\left(\b i^\a a^{(1-q)(i-1)}C\right)\\
	&\leq&\sum\limits_{i=0}^\infty \min\left\{M(A),\frac{1}{\b(i+1)^\a}a^{qi}\right\}+M(A)\sum\limits_{i=0}^\infty\Delta\left(\b(i+1)^\a a^{(1-q)i}C\right).
\end{eqnarray*}
Consider the case $M(A)\leq 1/(\b(i+1)^\a)a^{qi}$. A straightforward computation shows that it is equivalent to 
\[(i+1)\frac{q}{\a}\log\frac{1}{a}e^{(i+1)\frac{q}{\a}\log\frac{1}{a}}\leq\frac{\frac{q}{\a}\log\frac{1}{a}e^{\frac{q}{\a}\log\frac{1}{a}}}{(\b M(A))^\frac{1}{\a}},\mbox{ i.e.}\]
\[(i+1)\frac{q}{\a}\log\frac{1}{a}\leq W\left(\frac{\frac{q}{\a}\log\frac{1}{a}e^{\frac{q}{\a}\log\frac{1}{a}}}{(\b M(A))^\frac{1}{\a}}\right).\]
That is
\[i+1\leq g(M(A))\]
where
\[g(M(A)):=\frac{\a}{q\log\frac{1}{a}}W\left(\frac{\frac{q}{\a}\log\frac{1}{a}e^{\frac{q}{\a}\log\frac{1}{a}}}{(\b M(A))^\frac{1}{\a}}\right),\]
and therefore,
\begin{eqnarray*}
	\int\limits_A fdM&\leq&\sum\limits_{i>g(M(A))-1}\frac{1}{\b(i+1)^\a}a^{qi}+M(A)g(M(A))+M(A)\sum\limits_{i=0}^\infty\Delta\left(\b(i+1)^\a a^{(1-q)i}C\right)\\
	&\leq&\frac{a^{[g(M(A))]q}}{\b}\sum\limits_{i\geq 0}\frac{1}{(i+[g(M(A))]+1)^\a}a^{qi}+M(A)g(M(A))\\
	&&+M(A)\sum\limits_{i=0}^\infty\Delta\left(\b(i+1)^\a a^{(1-q)i}C\right)
\end{eqnarray*}
where $[g(M(A))]$ denotes the integer such that $g(M(A))-1<[g(M(A))]\leq g(M(A))$. Observe that
\begin{eqnarray*}
&&a^{[g(M(A))]q}\leq a^{(g(M(A))-1)q}=a^{-q}e^{-\a W\left(\frac{\frac{q}{\a}\log\frac{1}{a}e^{\frac{q}{\a}\log\frac{1}{a}}}{(\b M(A))^\frac{1}{\a}}\right)}\\
&=&a^{-q}\b M(A)\left(\frac{W\left(\frac{\frac{q}{\a}\log\frac{1}{a}e^{\frac{q}{\a}\log\frac{1}{a}}}{(\b M(A))^\frac{1}{\a}}\right)}{\frac{q}{\a}\log\frac{1}{a}e^{\frac{q}{\a}\log\frac{1}{a}}}\right)^\a = a^{-q}\b M(A)e^{-q\log\frac{1}{a}}g(M(A))^\a = \b M(A)g(M(A))^\a.
\end{eqnarray*}
Hence,
\[\frac{a^{[g(M(A))]q}}{\b}\sum\limits_{i\geq 0}\frac{1}{(i+[g(M(A))]+1)^\a}a^{qi}\leq M(A)\sum\limits_{i\geq 0}\left(\frac{g(M(A))}{i+[g(M(A))]+1}\right)^\a a^{qi}\leq\frac{M(A)}{1-a^q}.\]
Therefore,
	\[	\int\limits_A fdM\leq \frac{M(A)}{1-a^q}+M(A)g(M(A))+M(A)\sum\limits_{i=0}^\infty\Delta\left(\b(i+1)^\a a^{(1-q)i}C\right).\]
Now, using $W(xy)\leq W(x)+y$ and setting $\b:=1/a^q$, 
\begin{eqnarray*}
	M(A)g(M(A)) &=& \frac{M(A)\alpha}{q\log\frac{1}{a}}W\left(\frac{\frac{q}{\a}\log\frac{1}{a}e^{\frac{q}{\a}\log\frac{1}{a}}}{(\b M(A))^\frac{1}{\a}}\right)= \frac{M(A)\a}{q\log\frac{1}{a}}W\left(\left(\frac{\left(\frac{1}{a}\right)^{q}}{\b M(A)}\right)^\frac{1}{\a}\frac{q}{\a}\log\frac{1}{a}\right)\\
	&\leq&\frac{M(A)\a}{q\log\frac{1}{a}}W\left(\left(\frac{1}{M(A)}\right)^\frac{1}{\a}\right)+M(A), 
\end{eqnarray*}
which implies (ii).

(iii)  By Remark \ref{esr} and Lemma \ref{ucl}, for $M$-a.e. $\sigma\in\Sigma_G$,
\begin{eqnarray*}
  f(\sigma)&\leq& \sum\limits_{i=0}^\infty \Delta\left(d\left(w_{\sigma_{i}}\circ...\circ w_{\sigma_{1}}\circ F(\sigma), w_{\sigma_{i}}\circ...\circ w_{\sigma_{1}}x_{i(\sigma_1)}\right)\right)\\
&\leq& \sum\limits_{i=0}^\infty \Delta\left(a^id\left(F(\sigma), x_{i(\sigma_1)}\right)\right)\\
&\leq& \sum\limits_{i=0}^\infty \Delta\left(a^i\frac{d}{1-a}\right).
\end{eqnarray*}
Hence, by \eqref{fre}, since $M(\Sigma_G)=1$,
\[ \log Z_{1\infty}\leq\log |\S| +\frac{1}{\delta} \sum\limits_{i=0}^\infty \Delta\left(a^i\frac{d}{1-a}\right)\ M\mbox{-a.e.}\]
Thus (iii) follows.
\end{proof}

In the case when all maps of the CMS are contractions, we obtain the following lower bound for the DDM.

 \begin{Corollary}\label{fc}
Suppose $\M$ is non-degenerate such that all $p_e|_{K_{i(e)}}$'s are Dini-continuous,  there exists $\delta>0$ such that $p_e|_{K_{i(e)}}\geq\delta$ for all $e\in E$ and all $w_e|_{K_{i(e)}}$'s are contractions with a contraction rate $0<a<1$. Then 
\begin{eqnarray*}
\Phi(\lambda')(Q)&\geq& M(Q)\frac{1}{|\S|}e^{ -\frac{1}{\delta} \sum\limits_{i=0}^{\infty}\Delta\left(a^i\frac{d}{1-a}\right)}\ \ \ \mbox{ for all } Q\in \B(\Sigma).
\end{eqnarray*}
\end{Corollary}
\begin{proof}
By Corollary 1 (ii) in \cite{Wer11}, $M\in E(\M)$.  Thus the assertion follows by  Lemma \ref{lbl} (i) and Theorem \ref{acms} (iii).
\end{proof}

\section{A generalized construction}\label{gcs}

A simple way to obtain a positive set function which fixes the problem in \cite{Wer3} is through the following generalized construction.

\begin{Definition}
	Let $0<\beta\leq 1$. For $Q\subset\Sigma$, define
	\[\Phi_\beta(Q):=\inf\limits_{(A_m)_{m\leq 0}\in\C(Q)}\left(\sum\limits_{m\leq 0}\phi_0\left(S^mA_m\right)^\beta\right)^\frac{1}{\beta}.\]
\end{Definition}
However, we do not know whether this set function is a measure on $\B(\Sigma)$ for $0<\beta< 1$. Moreover, if it were, it would coincide with $\Phi$, as then
\begin{equation}\label{ubph}
  \Phi_\beta(Q)\leq\Phi_\beta\left(\bigcup\limits_{m\leq 0}A_m\right)\leq\sum\limits_{m\leq 0}\Phi_\beta(A_m)\leq\sum\limits_{m\leq 0}\phi_0(S^mA_m)
\end{equation}
for all $(A_m)_{m\leq 0}\in\C(Q)$, and, obviously, $\Phi_\beta\geq\Phi$. (Note that $\Phi_\beta(A)=\Phi(A)$  for all $A\in\bigcup_{m\leq 0}\A_m$ if $\phi_0\circ S^{-1}=\phi_0$, by Proposition 2 in \cite{Wer10}.)

Observe that only the countable subadditivity of $\Phi_\beta$ is used in \eqref{ubph}. The already known technique for obtaining countably subadditive set functions is the following.

\begin{Definition}
	Let $0<\beta\leq 1$. For $Q\subset\Sigma$, define
	\[\Phi^*_\beta(Q):=\inf\limits_{(A_m)_{m\leq 0}\in\C(Q)}\sum\limits_{m\leq 0}\Phi_\beta\left(A_m\right).\]
\end{Definition}
And, as expected, we obtain the following identity.
\begin{Proposition}
	Let $0<\beta\leq 1$. Then
	\[\Phi^*_\beta(Q)=\Phi(Q)\ \ \ \mbox{ for all }Q\subset\Sigma.\]
\end{Proposition}
\begin{proof}
	Let $Q\subset\Sigma$. Then, on one hand,
	\[\sum\limits_{m\leq 0}\Phi_\beta\left(A_m\right)\geq\sum\limits_{m\leq 0}\Phi\left(A_m\right)\geq\Phi\left(\bigcup\limits_{m\leq 0}A_m\right)\geq\Phi(Q),\]
	 and on the other,
	\[\Phi^*_\beta(Q)\leq\sum\limits_{m\leq 0}\Phi_\beta\left(A_m\right)\leq\sum\limits_{m\leq 0}\phi_0\left(S^mA_m\right)\]
	for all $(A_m)_{m\leq 0}\in\C(Q)$, which implies the assertion.
\end{proof}

Now, we show the absolute continuity of $M$ with respect to $\Phi_\beta$ for certain values of $\b$ in the setting of \cite{Wer3}.
\begin{Lemma}
  Suppose $\M$ is contractive with a contraction rate $0<a<1$,  $p_e|_{K_{i(e)}}$'s are Dini-continuous, and there exists $\d>0$ such that $p_e|_{K_{i(e)}}\geq\delta$ for all $e\in E$.  Suppose $M\in E(\M)$. Let $0<\beta<\d\log(1/a)/(1+\d\log(1/a))$ and $q:=\beta/((1-\beta)\d\log(1/a))$. Then
  \[\Phi_\beta(Q)\geq M(Q)^\frac{1}{\b}\frac{1}{|\S|}e^{-\frac{1}{\d}\left[\frac{1}{1-a^q}+\sum\limits_{i=0}^\infty\Delta\left(a^{(1-q)i}C\right)\right]}\ \ \ \mbox{ for all }Q\in\B(\Sigma).\]
\end{Lemma}
\begin{proof}
 Let $Q\in\B(\Sigma)$ and $(A_m)_{m\leq 0}\in\C(Q)$. Then, by the well-known inequality\\ $M(A)\log(M(A)/\phi_0(A))\leq\int_A\log ZdM$ (e.g. Lemma 4 (i) in \cite{Wer15}), Theorem \ref{acms} (i) and the choice of $\b$ and $q$,
 \begin{eqnarray*}
   &&\sum\limits_{m\leq 0}\phi_0\left(S^mA_m\right)^\b\\
   &\geq&\sum\limits_{m\leq 0,M(A_m)>0}M(A_m)e^{-\frac{1}{M(A_m)}\left(\b M(A_m)\log\frac{M(A_m)}{\phi_0\left(S^mA_m\right)}+(1-\b)M(A_m)\log M(A_m)\right)}\\
   &\geq&\sum\limits_{m\leq 0,M(A_m)>0}M(A_m)e^{-\frac{1}{M(A_m)}\left(\b\int\limits_{S^mA_m}\log ZdM+(1-\b)M(A_m)\log M(A_m)\right)}\\
   &\geq&\sum\limits_{m\leq 0,M(A_m)>0}M(A_m)e^{-\frac{\b}{M(A_m)}\left(M(A_m)\log |\S|+\frac{M(A_m)}{\delta}\left[\frac{1}{1-a^q}+\sum\limits_{i=0}^\infty\Delta\left(a^{(1-q)i}C\right)\right]\right)}\\
   &\geq&M(Q)\frac{1}{|\S|^\b}e^{-\frac{\b}{\d}\left[\frac{1}{1-a^q}+\sum\limits_{i=0}^\infty\Delta\left(a^{(1-q)i}C\right)\right]}.
 \end{eqnarray*}
 Thus
 \[\left(\sum\limits_{m\leq 0}\phi_0\left(S^mA_m\right)^\b\right)^\frac{1}{\b}\geq M(Q)^\frac{1}{\b}\frac{1}{|\S|}e^{-\frac{1}{\d}\left[\frac{1}{1-a^q}+\sum\limits_{i=0}^\infty\Delta\left(a^{(1-q)i}C\right)\right]},\]
 which implies the assertion.
\end{proof}

\end{document}